\documentclass[11pt,a4paper]{amsart}
\usepackage{amssymb,amsmath,amsthm,amsopn,amsfonts,amscd,enumerate,mathtools}
\usepackage[foot]{amsaddr}
\usepackage{tikz}
\usepackage{tikz-cd}
\usepackage{enumitem}
\usepackage{amsaddr}
\usepackage{etoolbox}
\usepackage{upgreek}\usetikzlibrary{matrix}
\usepackage[utf8]{inputenc}
\usepackage{hyperref}
\usepackage{thmtools,thm-restate}
\newcommand{\C}{\mathcal S}

\newcommand{\Aut}{\mathrm{Aut}}

\newcommand{\proset}{\,\mathrel{\lower 4pt\hbox{$\scriptscriptstyle/$}
\mkern -14mu\subseteq }\,}

\newcounter{thmcount}
\newtheorem{defn}{Definition}[section]

\newtheorem{lem}[defn]{Lemma}

\newtheorem{prop}[defn]{Proposition}
\newtheorem*{prop*}{Proposition}
\newtheorem{thm}[defn]{Theorem}
\newtheorem{cor}[defn]{Corollary}

\newtheorem*{claim*}{Claim}

\theoremstyle{remark}

\theoremstyle{remark}

\theoremstyle{remark}
\newtheorem{exmp}[defn]{Example}
\AtEndEnvironment{exmp}{\hfill$\Diamond$}
\theoremstyle{remark}

\theoremstyle{remark}

\theoremstyle{remark}

\theoremstyle{remark}

\theoremstyle{remark}

\theoremstyle{remark}
\newtheorem{rmk}[defn]{Remark}
\theoremstyle{remark}

\numberwithin{equation}{section}

 \makeatother
\title{Model-theoretic $K_1$ for modules over semisimple rings: (weak) Morita invariance}
\author{Sourayan Banerjee$^1$ and Amit Kuber$^2$}
\address{$^{1,2}$Department of Mathematics and Statistics\\ Indian Institute of Technology, Kanpur\\Uttar Pradesh-208016, India}
\email{$^1$sourayanbanerjee@gmail.com, $^2$askuber@iitk.ac.in (Corresponding author)}
\date{}
\date{} 
 
\begin{document}
\begin{abstract}
This paper is a sequel to a paper by the same authors, where they defined $K$-groups of model-theoretic structures, and computed $K_1$ of free modules over PIDs. In this paper, we compute $K_1$ of a right $M_q(R)$-module $M$, where $R$ is a division ring, $q\geq1$, and $|M_q(R)|\neq 2$. As a consequence, we obtain a (weak) Morita invariance $K_1(R_R)\cong K_1((M_q(R))_{M_q(R)})$ for all division rings $R$ and $q\geq 1$. Finally, we compute $K_1$ of a module over a semisimple ring by showing that the model-theoretic $K_1$ commutes with finite product of modules. We also show that the algebraic $K_1$ of a finite product of infinite matrix rings embeds into the model-theoretic $K_1$ of their right regular modules. 
\end{abstract}
\maketitle

\section{Introduction}

Let $L$ be a language and $M$ be a first-order $L$-structure. Motivated by, and extending Kraji\v{c}ek and Scanlon's definition of the Grothendieck ring $K_0(M)$ \cite{Kra}, the authors defined model-theoretic $K$-groups $K_n(M)$ for $n\geq 0$ in \cite{BK}. The latter definition employs Quillen's famous $S^{-1}S$-construction applied to the small symmetric monoidal groupoid $(\C(M),\sqcup,\emptyset)$, whose objects are definable subsets (with parameters) of finite powers of $M$, whose morphisms are definable bijections, and $\sqcup$ is the disjoint union.

Given a unital ring $R$, a right module $M_R$ could be thought of as a structure for the the language $L_R:=\langle+,-,0,\{\cdot_r\mid r\in R\}\rangle$, where $\cdot_r$ is unary function symbol describing the right action of the scalar $r$. After meticulous computations of $K_1(M_R)$ for free modules over some PIDs in \cite{BK}, we shift our attention to modules over semisimple rings in this paper. 

Our first major contribution is the following.
\begin{restatable}{theom}{A}\label{K1Final}
Suppose $R$ be a division ring, $q\geq 1$, $|M_q(R)|\neq2$, and $M$ is a non-zero $M_q(R)$-module. If $R^\times$ is the group of units in $R$, then
\begin{equation*}
K_1(M_{M_q(R)})\cong\begin{cases}\mathbb Z_2&\text{if }M\text{ is finite};\\\mathbb Z_2\oplus\bigoplus\limits_{i=1}^\infty\left(GL_{iq}(R)^{ab}\oplus\mathbb Z_2\right)\cong\mathbb Z_2\oplus\bigoplus\limits_{i=1}^\infty\left(\frac{R^\times}{[R^\times , R^\times]}\oplus\mathbb Z_2\right)&\text{otherwise}.\end{cases}
\end{equation*}
\end{restatable}

Recall that a unital ring $R$ is Morita equivalent to the matrix ring $M_q(R)$ for each $q\in \mathbb N$. A striking consequence of the above result is the following \textit{weak Morita invariance} for model-theoretic $K_1$ when $R$ is a division ring.
\begin{cor}\label{K1Morita}
Suppose $R$ be a division ring, $q\geq 1$, $|M_q(R)|\neq2$, and $M$ is a non-zero $M_q(R)$-module. Then $K_1({M}_{M_q(R)})\cong K_1(M_R)$. In particular, $K_1(R_R)\cong K_1((M_q(R))_{M_q(R)})$ for every division ring $R$ and $q\geq 1$.
\end{cor}

Recall that the algebraic $K_1$-group of a unital ring $R$ , denoted $K_1^\oplus(R)$, is isomorphic to $K_1^\oplus(M_q(R))$ \cite[Example~III.1.1.4]{Weibel} since $R$ and $M_q(R)$ are Morita equivalent. Moreover, this isomorphism is functorially induced by a categorical equivalence $F_q:\mathrm{Proj}(R) \rightarrow \mathrm{Proj}(M_q(R))$ between their respective categories of finitely generated projective modules. On the contrast, the isomorphism in Corollary \ref{K1Morita} between model-theoretic $K_1$-groups is not functorial; therefore, we call it a weak Morita invariance of model-theoretic $K_1$.

% SIMILAR RESULTS FOR $K_0$??

Our second major contribution is the following theorem which states that model-theoretic $K_1$ commutes with finite products of modules when the ring is semisimple.
\begin{restatable}{theom}{B}\label{K1semisimple} 
Let $S$ be a unital semisimple ring written as $\prod_{i=1}^k S_i$, where for each $i$, $S_i:=Se_i=M_{q_i}(R_i)$ for a division ring $R_i$, idempotent $e_i$, and $q_i\geq 1$ (thanks to the Wedderburn-Artin theorem). Write a right $S$-module $M_S$ as $\prod_{i =1}^k{(M_i)}_{M_{q_i}(R_i)}$, where $M_i:=Me_i$ is a right $M_{q_i}(R_i)$-module. Assume that each $M_i$ is infinite. Then $$K_1(M_S)  \cong \prod_{i=1}^kK_1\left({(M_i)}_{M_{q_i}(R_i)}\right).$$
\end{restatable}
This result is a partial model-theoretic analogue of a similar result for algebraic $K_1$ \cite[Example~III.1.1.3]{Weibel} which states that the latter commutes with finite products.

As a consequence of the two theorems above, using the notations of Theorem \ref{K1semisimple} and under its hypotheses, if each $S_i$ is infinite, then we show (Corollary \ref{algmodss}) that $K_1^\oplus(S)$ naturally embeds into $K_1(S_S)$.

The proofs of Theorems \ref{K1Final} and \ref{K1semisimple} use a recipe similar to that in \cite{BK}, albeit with necessary modifications to deal with non-commutative rings (\S~\ref{K1div}, \ref{K1matdiv}), and vector dimensions for definable sets (\S~\ref{K1ss}).

Theorem \ref{K1Final} generalizes \cite[Theorem~5.2]{BK}, and corrects the following two errors in \cite[Corollaries~7.10,7.11]{BK}.
\begin{itemize}
    \item The authors mistakenly stated $K_1(V_F)\cong K_1(F_F)$ for an infinite vector space over a finite field $F$ other than $F_2$; however, the former is infinite while the latter is (isomorphic to) $\mathbb Z_2$.
    \item The theory $Th(V_F)$ of an infinite vector space $V_F$ over a finite field $F$ is closed under products as a consequence of \cite[Lemma~1.2.3]{PreBk} but in those corollaries the authors erroneously used that this condition fails. In fact, the correct corollaries follow from \cite[Theorem~7.6]{BK} instead of \cite[Theorem~7.8]{BK}.
\end{itemize}

The rest of the paper is organised as follows. We recall the definition of model-theoretic $K$-groups for modules in \S~\ref{MTM} along with Bass' formula for computing $K_1$ (Theorem \ref{K1Bass}). Since semisimple rings form a subclass of the class of von Neumann regular rings, we recall the fundamental results in the model theory of modules over such rings in \S~\ref{MThSS}, notably Theorem \ref{quantfree} and Proposition \ref{funR} stating elimination of quantifiers and  of imaginaries respectively. The proof of Theorem \ref{K1Final} for the case of modules over division rings and over matrix rings is completed in \S~\ref{K1div} and \S~\ref{K1matdiv} respectively. Finally, Theorem \ref{K1semisimple} is proved in \S~\ref{K1ss} along with a generalization to certain modules over von Neumann regular rings (Theorem \ref{partition}).

As is the convention in logic, the set $\mathbb N$ of natural numbers includes $0$. For group-theoretic preliminaries regarding semi-direct products, wreath products, finitary symmetric groups on a countable set, and their abelianizations, we refer the interested reader to \cite[\S~2]{BK}.

\section{Model-theoretic K-groups of modules}\label{MTM}
Quillen's definition of $K$-groups of a skeletally small symmetric monoidal groupoid $(\C,\ast,e)$ uses his famous $\C^{-1}\C$ construction (see \cite[\S~IV.4]{Weibel} for more details). He defined the $K$-\emph{theory space} $K^\ast(\C)$ of $\C$ to be the geometric realization of $\C^{-1}\C$ and the $K$-groups of $\C$ as $K_n^\ast(\C):=\pi_nK^\ast(\C)$. We used this construction to associate K-theory to model-theoretic structures in \cite{BK}, which we recall below.

Let $L$ be a language, $M$ a first-order $L$-structure and $m\geq1$. By an abuse of notation, we denote the domain of the structure $M$ again with the same notation. We will always assume that definable means definable with parameters from the universe. For each $m \geq 1$, let $\text{Def}(M^m)$ be the collection of all definable subsets of $M^m$, and set $\overline{\text{Def}}(M):= \bigcup_{m\geq 1} \text{Def}(M^m)$.

Let $\C(M)$ denote the groupoid whose objects are $\overline{\mathrm{Def}}(M)$ and morphisms are definable bijections between definable sets, i.e., bijections whose graphs are definable sets. Note that $(\C(M),\sqcup,\emptyset)$ is a symmetric monoidal groupoid, where $\sqcup$ is the disjoint union. Moreover, the Cartesian product of definable sets induces a pairing (see \cite[\S~IV.4]{Weibel} for a definition) on this monoidal category.
%Say that $A \subseteq M^m$ is \emph{definable with parameters} if there is a formula $\phi(x_1,...x_m,y_1,...,y_n)$ such that for all $a_1,...,a_m \in M$, $(a_1,...,a_m) \in A$ if and only if  $M\vDash \phi[a_1,...,a_m,b_1,...,b_n]$ for some $b_1,...,b_n \in M$. 

\begin{defn}\cite[Definition~4.5]{BK}
Define $K_n(M):=K_n^\sqcup(\C(M))$ for each $n\geq 0$.
\end{defn}

% \begin{defn}
% Say that two definable sets $A,B\in\overline{\mathrm{Def}}(M)$ are \emph{definably isomorphic} if there exists a definable bijection between them. Definable isomorphism is an equivalence relation on $\overline{\mathrm{Def}}(M)$ and the equivalence class of a definable set $A$ is denoted by $[A]$. We use $\widetilde{\mathrm{Def}}(M)$ 
% to denote the set of all equivalence classes with respect to this relation.
% \end{defn}

% The assignment  $A\mapsto[A]$ defines a surjective map $[-]:\overline{\mathrm{Def}}(M)\rightarrow\widetilde{\mathrm{Def}}(M)$. We can regard $\widetilde{\mathrm{Def}}(M)$ as an $L_{ring}$-structure. In fact, it is a semiring with respect to the operations defined as follows:
% \begin{itemize}
% 	\item $0 := [\emptyset]$;
% 	\item $1 := [\{*\}]$ for any singleton subset $\{*\}$ of $M$;
% 	\item $[A]+[B] := [A'\sqcup B']$ for $A'\in[A],B'\in[B]$ such that $A'\cap B'=\emptyset$; and
% 	\item $[A]\cdotp[B] := [A\times B]$.
% \end{itemize}

\begin{rmk}
The Grothendieck ring $K_0(M)$ defined this way coincides with that of \cite{Kra}.
\end{rmk}

\begin{exmp}\label{finitestrKtheory}
It follows from Barratt-Priddy-Quillen-Segal theorem \cite[Theorem~IV.4.9.3]{Weibel}\label{BPQS} that if $M$ is a finite structure with at least two elements then $K_n(M)\cong\pi^s_n$, where $\pi^s_n$ is the $n^{th}$ stable homotopy group of spheres. In particular, $K_0(M)\cong\mathbb Z$ and $K_1(M)\cong\mathbb Z_2$.
\end{exmp}

\begin{rmk}\label{trfaith}
For a model-theoretic structure $M$, \emph{translations are faithful} in $\C(M)$, i.e., for all $A,B\in\C(M)$, the translation $\mathrm{Aut}_{\C(M)}(A)\to\mathrm{Aut}_{\C(M)}(A\sqcup B)$ defined by $f\mapsto f\sqcup id_B$ is an injective map. 
\end{rmk}

Bass was the first to introduce the group $K_1$, and in view of the above remark, the following theorem could be used for the computation of $K_1(M)$.
\begin{thm}\cite{Bass}\label{K1Bass}
Suppose $\C$ is a symmetric monoidal groupoid with faithful translations. Then $K_1(\C)\cong\varinjlim\limits_{s\in \C}(\mathrm{Aut}_{\C}(s))^{ab}$.
\end{thm}

\begin{rmk}\label{ctblcofinal}
Suppose that $(\C,\ast,e)$ is a symmetric monoidal groupoid whose translations are faithful. Further suppose that $\C$ has a countable sequence of objects $s_1,s_2,\hdots$ such that $s_{n+1}\cong s_n\ast a_n$ for some $a_n\in \C$, and satisfying the cofinality condition that for every $s\in \C$ there is an $s'$ and an $n$ such that $s\ast s'\cong s_n$. In this case, we can form the colimit $\mathrm{Aut}(\C):=\varinjlim_{n \in \mathbb{N}}\mathrm{Aut}_{\C}(s_n)$, and hence $K_1(\C)=(\mathrm{Aut}(\C))^{ab}$.
\end{rmk}

Every right $R$-module $M$ is a first-order structure for the language $L_R$ of right $R$-modules, where $L_{R} := \langle +,-,0,\{\cdot_r:r \in R\} \rangle$ with $\cdot_r$ a unary function symbol for the scalar multiplication by $r\in R$ on the right. The right $R$-module structure $M$ will be denoted as $M_R$.

The theory of the $L_R$-structure $M_R$ admits a partial elimination of quantifiers with respect to positive primitive ($pp$) formulas (see \cite[\S~2.1]{PreBk} for a definition) as a consequence of the fundamental theorem of the model theory of modules due to Baur and Monk (\cite{Baur}). 

Say that a subset $B$ of $M^m$ is \emph{$pp$-definable} if it is definable by a $pp$-formula, and a \emph{$pp$-definable function} is a function between two definable sets whose graph is $pp$-definable. Every $pp$-definable set is either the empty set or a coset of a $pp$-definable subgroup of $M^n$. Furthermore, the conjunction of two $pp$-formulas is (logically equivalent to) a $pp$-formula.

Let $\mathcal L_n(M_R)$ (or just $\mathcal{L}_n$, if the module is clear from the context) denote the meet-semilattice of all $pp$-definable subsets of $M^n$ under intersection. Set $\overline{\mathcal {L}}(M_R):=\bigcup_{n\geq 1}\mathcal L_n(M_R)$. Let $\mathcal L_n^\circ(M_R)$ be the sub-meet-semilattice of $\mathcal L_n(M_R)$ consisting only of the $pp$-definable subgroups. The notation $\overline{\mathcal{X}}(M_R)$ (or just $\overline{\mathcal X}$, if the module is clear from the context) will denote the set of \emph{colours}, i.e., $pp$-definable bijection classes of elements of $\overline{\mathcal L}(M_R)$. For $A\in\overline{\mathcal L}$, the notation $[[A]]$ will denote its $pp$-definable bijection class. The set $\overline{\mathcal{X}}^*:=\overline{\mathcal{X}}\setminus[[\emptyset]]$ of non-trivial colours is a monoid under multiplication induced by Cartesian product.

\begin{defn}
The theory of the module $M_R$ is said to be \emph{closed under products} if for each $n\geq 1$ and for any subgroups $A,B\in\mathcal L_n$, the index $[A:A\cap B]$ is either $1$ or $\infty$.
\end{defn}

The next result follows immediately from \cite[Lemma~1.2.3]{PSL}.
\begin{prop}\label{TequalsTaleph0}
The following hold for an infinite free module $M_R$ over a ring $R$:
\begin{enumerate}
    \item $\mathcal L_n^\circ(M_R)\cong\mathcal L_n^\circ(R_R)$;
    \item if $R$ is finite, then the theory of $M_R$ is closed under products.
\end{enumerate}
\end{prop}

The second author computed $K_0$ for all modules; we mention only a special case below.
\begin{thm}\cite[Theorem~4.1.2]{Kuber1}\label{GrothRing}
Suppose $M_R$ is a module whose theory is closed under products. Then $K_0(M_R)$ is isomorphic to the monoid ring $\mathbb Z[\overline{\mathcal{X}}^*(M_R)]$.
\end{thm}

\section{Model theory of modules over infinite semisimple rings}\label{MThSS}
Recall that a unital ring $R$ is \textit{von Neumann regular} if for every element $a \in R$ there exists an element $r \in R$ such that $ara=a$. A few notable examples of von Neumann regular rings include fields, division rings (ring where any non-zero element is a unit), semisimple rings, and the endomorphism ring of $F$-linear morphisms, $\mathrm{End}_{F}(V)$, for any vector space $V$ over a field $F$.
\begin{rmk}\label{vNRop}
The class of von Neumann rings is closed under finite direct products and opposite rings.
\end{rmk}
The class of von Neumann regular rings can be completely characterized model-theoretically as described in the next result--several cases of this result were proven by multiple authors over a period of few decades but we only cite a book.
\begin{thm}\cite[A.2.1]{Hodges}\label{quantfree}
A ring $R$ is von Neumann regular iff every $pp$-formula in the language $L_R$ is equivalent to one without quantifiers.
\end{thm}

Moreover, when $R$ is von Neumann regular, we also get a complete elimination of $pp$-imaginaries--this result is stated as \cite[Proposition~10.2.38]{PSL} in functor-category-theoretic language. The next statement is an algebraic consequence of this result that provides a complete description of all $pp$-definable subsets of $R_R^n$ as well as of $pp$-definable bijections between them.
\begin{prop}\label{funR}
If $R$ is von Neumann regular, and $f:D_1\to D_2$ is a $pp$-definable bijection between $pp$-definable subsets of $R_R^n$ for some $n\geq1$, then both $D_1$ and $D_2$ are cosets of right $R$-submodules of $R_R^n$, and there are $\overline d_i\in D_i$ and $A\in GL_n(R)$ such that for each $\overline x\in D_1$, we have $f(\overline x)=(\overline x-\overline d_1)A+\overline d_2$.
\end{prop}

Our main object of study is the class of \textit{semisimple rings}, i.e., the class of von Neumann regular rings satisfying the (equivalent) conditions of the next theorem.
\begin{thm}\label{smismple}\cite[\S~4.1,4.2]{Rot}
The following are equivalent for a ring $R$.
\begin{enumerate}
    \item The ring $R$ is Noetherian and von Neumann regular.
    \item All right $R$-modules are projective.
    \item All right $R$-modules are injective.
    \item (Wedderburn-Artin) There are division rings $R_1,\cdots,R_k$ such that $R\cong \prod\limits_{i=1}^kM_{q_i}(R_i).$
\end{enumerate} 
\end{thm}
From the perspective of the computation of model-theoretic $K_1$, the above characterization of semisimple rings demands that we first need to compute $K_1(M_R)$, where $R$ is a division ring or a matrix ring over a division ring. We already addressed the computation of $K_1(M_R)$, where $R$ is a finite division ring, or equivalently, a finite field, in \cite[\S~7]{BK}. Therefore, in this paper, we focus our attention to infinite division rings.

We require some standard properties of modules over a division ring $R$.
\begin{prop}\label{dvprop}\begin{enumerate}
    \item Finitely generated modules over a division ring $R$ satisfy the invariant basis property, i.e., isomorphic finitely generated modules have equal rank.
  %  \item Since $D$ is non-commutative, the ring of endomorphisms for a right $D$-module $M$, $\mathrm{End_D}(M)$ only consists of $D^{op}$-linear morphisms.
    \item There is an obvious isomorphism $(M_q(R))^{op} \cong M_q(R^{op})$ of matrix rings which restricts to a group isomorphism $(GL_q(R))^{op}\cong GL_q(R^{op}).$   
\end{enumerate}
\end{prop}

Since a division ring $R$ is Noetherian and every $R$-module is free, thanks to Proposition \ref{TequalsTaleph0}(1), the proof of \cite[Proposition~6.1]{BK}, that does not use commutativity, could be adapted to obtain the following.
\begin{thm}\label{divth}
If $R$ is a division ring and $M_R$ is an infinite module then $\overline{\mathcal{X}}^*(M_R) \cong \mathbb N$.
\end{thm}
% \begin{proof}
% Every division ring is Noetherian. Hence the $pp$-definable subgroups of $R_R$ are right ideals (see \cite[Exercise~2, p.19]{PreBk}), i.e., either $\{0\}$ or $R$. Moreover, since every $R$-module is free, thanks to Proposition \ref{dvprop}(1), a pp-definable subgroup of $R_R^n$ is isomorphic to $R_R^k$ for $k\leq n$.
% \end{proof}

% \section{Morita equivalence between $M_q(R)$ and $R$}
% The purpose of this section is to recall the Morita equivalence. Specifically if $R$ is a division ring then we show the relation between the theory of $R$ and $M_q(R)$.
Recall that a unital ring $R$ is Morita equivalent to the matrix ring $M_q(R)$ for any $q\geq 1$, i.e., there is an equivalence $F_q:\mathrm{Mod}\text{-}R\to\mathrm{Mod}\text{-}M_q(R)$ between the module categories.
\begin{rmk}\label{chad}
Every right ideal of $M_q(R)$ is isomorphic to $F_q(N)$ for a submodules $N$ of $R_R^q$. In particular, the $M_q(R)$-module $M_{1\times q}(R)$ is isomorphic to $F_q(R_R)$. Moreover, if $R$ is a division ring, then $F_q(R_R^k)$ is a free module over $R$ with rank $kq$ thanks to \ref{dvprop}(1). In particular, if $M_{M_q(R)}$ is finite then so is $R$.
\end{rmk}

This remark together with Propositions \ref{TequalsTaleph0}(1) and \ref{funR} ensure that the proof of Theorem \ref{divth} could be readily adapted to yield the following.
\begin{thm}\label{matdiv}
Suppose $R$ is a division ring and $M_{M_q(R)}$ is infinite. Then $\overline{\mathcal{X}}^*(M_{M_q(R)})\cong q\mathbb N\cong \mathbb N$.
\end{thm}
The above theorem together with Theorem \ref{GrothRing} and Example \ref{finitestrKtheory} yields the following.
\begin{cor}[Weak Morita invariance of model-theoretic $K_0$]
If $R$ is a division ring, $q\geq 1$, and $M$ is an infinite $M_q(R)$-module, then $K_0({M}_{M_q(R)})\cong K_0(M_R)$. As a consequence, for all division rings $R$ and $q\geq1$, we have $K_0((M_q(R))_{M_q(R)})\cong K_0(R_R)$.
\end{cor}

\begin{rmk}\label{K0weakMorita}
The above corollary is a weak Morita invariance of model-theoretic $K_0$, and not the usual Morita invariance as seen in its algebraic counterpart \cite[Corollary~II.2.7.1]{Weibel} since the isomorphism $K_0(R_R)\cong K_0(M_q(R)_{M_q(R)})$ is not induced by the functor $F_q$ between module categories when $R$ is infinite. Indeed, the functor $F_q$ naturally yields a bijective map from $\mathcal{L}_k^\circ(R_R)$ to $\mathcal L_k^\circ(M_q(R)_{M_q(R)})$ thanks to Proposition \ref{funR}, but the latter bijection fails to extend to a bijective map from $\mathcal{L}_k(R_R)$ to $\mathcal L_k(M_q(R)_{M_q(R)})$ since there are far too many parameters on the right side compared to the left side.
\end{rmk}

\section{$K_1$ of modules over a division ring}\label{K1div}
The goal of this section is to prove Theorem \ref{K1Final} when $q=1$, i.e., the computation of $K_1(M_R)$ for a non-zero right $R$-module $M_R$ for a division ring $R$. Note that $M$ is a free $R$-module. We may assume that $M$ is infinite for otherwise the conclusion follows from Example \ref{finitestrKtheory}. 

Recall from Proposition \ref{TequalsTaleph0}(2) that the theory of $M_R$ is closed under products. Moreover, Theorem \ref{divth} yields that $\overline{\mathcal{X}}^*(M_R) \cong \mathbb N$. The proof for this case of the theorem is along lines similar to the computation of $K_1(V_F)$ \cite[\S~5]{BK}, where $V_F$ is an infinite vector space over an infinite field $F$. We follow all three steps of the proof of the latter while highlighting the changes for the division ring case. For brevity, we denote $\C(M_R)$ by $\C$ and $\overline{\mathcal{X}}^*(M_R)$ by $\overline{\mathcal{X}}^*$.
% \begin{thm}\label{K1FINAL}
% Suppose $M_R$ is a non-zero right $R$-module over a division ring $R$, and $M$ is finite if $R\cong\mathbb F_2$. If $R^\times$ is the group of units in $R$, then
% \begin{equation*}
% K_1(M_R)\cong K_1(R_R)\cong\begin{cases}\mathbb Z_2&\text{if }M\text{ is finite};\\\mathbb Z_2\oplus\bigoplus\limits_{i=1}^\infty\left(GL_i(R)^{ab}\oplus\mathbb Z_2\right)\cong\mathbb Z_2\oplus\bigoplus\limits_{i=1}^\infty\left(\frac{R^\times}{[R^\times , R^\times]}\oplus\mathbb Z_2\right)&\text{otherwise}.\end{cases}
% \end{equation*}
% \end{thm}

\noindent{\textbf{Step I:}} In this step, we associate a ``dimension'' $\dim(f)$ to each automorphism $f$ of a definable set through its ``support'', and show that the groups of bounded-dimension automorphisms of sufficiently large definable sets are isomorphic.

Recall the definition of dimension of a definable set from \cite[Definition~6.3]{BK}: for $E\in\C$, set $\dim(E) :=\begin{cases}-\infty&\text{if }E=\emptyset;\\ \max\{\mathfrak{A}\in\overline{\mathcal{X}}^*\mid\Lambda_{\mathfrak{A}}(E) \neq 0\}&\text{otherwise,}\end{cases}$ where $\Lambda_{\mathfrak A}$ is a definable-bijection-invariant integer-valued function defined in \cite[\S~5.2]{Kuber1} that takes value $0$ at all but finitely many inputs. In other words, if $E\neq\emptyset$, then $\dim(E)$ is the degree of the polynomial $[E]\in K_0(M_R)\cong\mathbb Z[\mathbb N]\cong\mathbb Z[X]$.

\begin{defn}\label{supp}
Let $E\in\C$ and $f\in\Aut_\C(E)$. The \emph{support of $f$} is the (definable) set $\operatorname{Supp}(f):=\{a\in E: f(a)\neq a\}$. Set $\dim(f):=\dim(\operatorname{Supp}(f))$.
\end{defn}
The main result in this step is the following.
\begin{prop}\cite[Proposition ~5.4]{BK}\label{autinftydim}
For $E\in\C$, let $\Omega_m(E):=\{f\in\Aut_\C(E):\dim(f)\leq m\}$ be the subgroup of $\Aut_\C(E)$ of automorphisms fixing all elements of $E$ outside a subset of dimension at most $m$. If $E_1,E_2\in\C$ have dimension strictly greater than $m$, then $\Omega_m(E_1)\cong\Omega_m(E_2)$.
\end{prop}

\noindent{\textbf{Step II:}} We first recall the basic notations for the reader's ease. For all $n \in \mathbb N$, set $\Omega_n^n:=\Aut_{\C}(M^n)$. For each $0\leq m<n$, let $\Omega^n_m:=\Omega_m(M^n)$ and $\Sigma^n_m$ denote the finitary permutation group on a countable set of cosets of an $m$-rank submodule of $M^n$. For each $n\geq 1$, let $\Upsilon^n:=\Upsilon^n(M_R)$ denote the subgroup of $\Aut_{\C}(M^n)$ consisting only of $pp$-definable bijections. Since  the group of $R$-linear automorphisms of $R_R^n$ is $ GL_n(R)^{op} \cong GL_n(R^{op})$ (see Proposition \ref{dvprop}(2) for the last isomorphism), it follows from Proposition \ref{funR} that a $pp$-definable bijection is in fact a definable linear bijection and that $\Upsilon^n$ is the group $GL_n(R^{op})\ltimes M^n$, where $GL_n(R^{op})$ acts on $M^n$ on the right by matrix multiplication. Furthermore, subgroups $\Upsilon^n_m:=\Upsilon^m\wr\Sigma^n_m$ of $\Omega^n_m$ for $1\leq m<n$ satisfy $\Omega^n_m\cong\Upsilon^n_m\ltimes\Omega^n_{m-1}$ as shown in Step II of \cite[\S~5]{BK} using Proposition \ref{autinftydim}. The rest of the proof of Step II there follows verbatim to conclude 
$$\Omega^n_n\cong\Upsilon^n\ltimes(\Upsilon^n_{n-1}\ltimes(\Upsilon^n_{n-2}\ltimes(\cdots(\Upsilon^n_1\ltimes\Omega^n_0)\cdots))).$$
\noindent{\textbf{Step III:}} In this final step, we first compute $(\Omega{^n_n})^{ab}$ exactly as in Step III of \cite[\S~5]{BK} except for a very subtle change where we replace $GL_n(R)$ with $(GL_n(R))^{op}$ as explained in the step above.
\begin{prop}
If $R$ is a division ring, $|R|\neq2$, and $M$ is infinite, then for each $n\geq 1$, we have
\begin{equation*}
(\Omega^n_n)^{ab}\cong(GL_n(R))^{ab}\oplus\bigoplus_{i=0}^{n-1}\left((GL_i(R))^{ab}\oplus\mathbb Z_2\right)\cong \frac{R^\times}{[R^\times, R^\times]}\oplus\bigoplus_{i=1}^{n-1}\left(\frac{R^\times}{[R^\times, R^\times]}\oplus\mathbb Z_2\right)\oplus\mathbb Z_2.
\end{equation*}
\end{prop}
\begin{proof}
The proof of the first isomorphism follows that of \cite[Proposition~5.5]{BK} verbatim since $GL_n(R^{op})\cong GL_n(R)^{op}$. For the second isomorphism, we use a result of Dieudonn\'{e} from 1943 where he explicitly proved that $(GL_n(R))^{ab} \cong \frac{R^\times}{[R^\times,R^\times]}$ for all $n \geq 1$ (except for $n =2$, when $|R| = 2$) \cite[III.1.2.4]{Weibel}.
\end{proof}
Recall from Theorem \ref{K1Bass} and Remark \ref{ctblcofinal} that $K_1(M_R)\cong\varinjlim_{n\in \mathbb N}(\Omega^n_n)^{ab}$, and hence the proof of thsi case of Theorem \ref{K1Final} is complete thanks to the above proposition.

\begin{rmk}\label{step12}
If $R$ is a von Neumann regular ring and $M_R$ is an infinite right $R$-module such that the theory of $M_R$ is closed under products and $\overline{\mathcal{X}}^*(M_R) \cong \mathbb{N}$, then Steps I and II of the proof of Theorem \ref{K1Final} follow verbatim.
\end{rmk}

Let us recall the computation of $K_1^\oplus(R)$ from algebraic K-theory.
\begin{thm}\cite[III.1.2.4]{Weibel}
If $R$ is a division ring, then $K_1^\oplus(R)\cong\frac{R^\times}{[R\times,R^\times]}$.
\end{thm}

There is a beautiful connection between algebraic and model-theoretic $K_1$-groups whose proof follows that of \cite[Theorem~9.1]{BK} verbatim.
\begin{thm}\label{algmoddiv}
For an infinite division ring $R$, there is a natural embedding of $K_1^\oplus(R)$ into $K_1(R_R)$ induced by the inclusion functor $\mathrm{Free}(R)\to\C(R_R)$ (thanks to Proposition \ref{funR}), where $\mathrm{Free}(R)$ is the full subcategory of $\mathrm{Mod}\mbox{-}R$ consisting of finitely generated free modules.
\end{thm}

\begin{rmk}\label{algmoddivrmk}
Suppose $GL(R):=\varinjlim_{n\in\mathbb N} GL_n(R)$ for an infinite division ring $R$. Then the composition $GL(R) \twoheadrightarrow K_1^\oplus(R)\hookrightarrow K_1(R_R)$ can be described as follows: if $A \in GL_n(R)$ is not in the image of the natural embedding of $GL_{n-1}(R)$ into $GL_n(R)$, then it maps to $\det(A)\in(GL_n(R))^{ab}$, where $(GL_n(R))^{ab}$ is the leading term of $(\Omega^n_n)^{ab}$.
\end{rmk}

\section{$K_1$ of $M_q(R)$-modules for a division ring $R$}\label{K1matdiv}
The main goal of this short section is to prove Theorem \ref{K1Final} in its full generality by computing $K_1(M_{M_q(R)})$ when $M_q(R)$ is the matrix ring over a division ring $R$, $q\geq 1$, and $|M_q(R)|\neq2$. 
% \begin{thm}(Weak Morita invariance of model-theoretic $K_1$)\label{K1Morita}
% Let $R$ be a division ring and $q\geq 1$, except $R= \mathbb F_2$ and $q =1$, and $M$ be an $M_q(R)$-module. Then $K_1({M}_{M_q(R)}) \cong K_1(M_R)$.
% \end{thm}

The proof is divided into three steps similar to the proof in the section above. Thanks to the last sentence of Remark \ref{chad} and Example \ref{finitestrKtheory}, we assume that $M$ is infinite. We argued in Theorem \ref{matdiv} that $\overline{\mathcal{X}}^*(M_{M_q(R)})\cong q\mathbb N\cong \mathbb N$, and noted in Proposition \ref{TequalsTaleph0}(2) that the theory of $M_{M_q(R)}$ is closed under products. Therefore, Remark \ref{step12} yields that we only need to deal with Step III.

The next result computes $(\Upsilon^n(M_{M_q(R)}))^{ab}$ for most values of $q$ and $n$, and it's proof is essentially that of \cite[Lemma~7.1]{BK}--the latter result is stated for commutative rings but its proof only uses elementary matrices and does not depend on the commutativity of the ring.
\begin{lem}\label{quotfree}
Suppose $R$ is a unital ring and $M_R$ is a right $R$-module. Then for each $q\geq2$ we have $(GL_q(R)^{op} \ltimes M^q)^{ab} \cong (GL_q(R))^{ab}$. Moreover, if the multiplicative identity $1$ in $R$ can be written as a sum of two units, then the conclusion also holds true for $q=1$.
\end{lem}

The hypothesis of the above lemma fails when $|R|=2$ and $q=1$. For all other cases, combining Dieudonne's result and Lemma \ref{quotfree} with the fact that $GL_k(M_q(R)) \cong GL_{kq}(R)$ for all $k,q \in \mathbb N$, we can follow the computations in Step III of \S~5 to conclude the proof of Theorem \ref{K1Final}.

\begin{rmk}\label{K1weakMorita}
The isomorphism $K_1(R_R)\cong K_1((M_q(R))_{(M_q(R))})$ for each division ring $R$ stated in Corollary\ref{K1Morita} could be interpreted as weak Morita invariance of model-theoretic $K_1$, and not the usual Morita invariance as seen in its algebraic counterpart \cite[Example~III.1.1.4]{Weibel}--this property is shared with model-theoretic $K_0$ as explained in Remark \ref{K0weakMorita}. When $R$ is infinite, we overcome the difference between the sets of parameters for different rings thanks to Lemma \ref{quotfree} as well as the use wreath products with finitary permutation groups in Step II of the proof via Proposition \ref{autinftydim}.
\end{rmk}

% Now we focus on the case where $R = \mathbb F_q$, a finite field of order $q \neq 2$. The reason that we compute this separately is because for finite fields we have $T \neq T^{\aleph_0}$. But 
% For the $T \neq T^{\aleph_0}$ case we observe that for $M_{M_n(R)}$, a free right $M_n(R)$-module the following holds verbatim
% \begin{prop}\label{Upsi}
%     Let $\mathbf G_k$ be the set of all $pp$-definable finite-index subgroups of $M^k_{M_n(R)}$. Then $(\mathbf G_k,\leq)$ is a poset under the subgroup relation in such a way that $(\mathbf G_k,\leq)^{op}$ is a directed set. Then 
%     \begin{eqnarray*}
%     \Upsilon^n(M_{M_n(R)})&\cong&\varinjlim_{G \in (\mathbf G_k,\leq)^{op}} (\mathrm{Aut}_{\mathcal{L}}(G)\wr \FS(M^n/G))\\
%     &\cong&\varinjlim_{G \in (\mathbf G_k,\leq)^{op}} \left((GL_n(R)^{op}\ltimes M^n)\wr \FS(M^n/G)\right).  
%     \end{eqnarray*}
% \end{prop}
% For a detailed discussion on Proposition \ref{Upsi} we refer to \cite[\S~6]{BK}.
% If $R$ is any finite field of order $\neq 2$ then $M_n(R)$ will always be such a ring that $1=u+v$ for some units $u,v \in GL_n(R)$. Hence, as showed in \cite[Theorem~7.8]{BK}, here using similar steps along with Lemma \ref{quotfree}, Proposition \ref{Upsi}, and observing $GL_k(M_n(R)) = GL_{kn}(R)$ we arrive at the following conclusion

% \begin{thm}\label{TneqT0}
% Let $R$ be a finite field of order $\neq 2$. Suppose that the theory $T$ of $M_R$ satisfies $T \neq T^{\aleph_0}$. If $(\mathbf G_1,\leq)^{op}$ contains a cofinal system of even-indexed subgroups of $M$ then $$K_1(M_{M_n(R)})\cong K_1(M_R) $$
% \end{thm}

\section{$K_1$ of modules over semisimple rings}\label{K1ss}
Throughout this section, $S$ will denote a unital semisimple ring. Thanks to the Wedderburn-Artin theorem (Theorem \ref{smismple}(4)), we have $S\cong\prod_{i=1}^kS_i$, where $S_i:=M_{q_i}(R_i)$ for some division ring $R_i$ and $q_i\geq 1$. Let $e_i$ be the idempotent such that $Se_i=M_{q_i}(R_i)$. 
\begin{rmk}\label{composemi}
Theorem \ref{smismple}(2) yields that a right $S$-module $M_S$ can be written as $M\cong\prod_{i=1}^k M_i$, where $M_i:=Me_i$ is a right $S_i$-module. As a consequence, the category of $\mathrm{Mod}$-$S=\mathrm{Mod}$-$(\prod_{i=1}^k S_i)$ is equivalent to the category of $\prod_{i=1}^k(\mathrm{Mod}$-$S_i)$.
\end{rmk}

Assume that each $M_i$ is infinite. The main goal of this section is to prove Theorem \ref{K1semisimple}. As in \S~\ref{K1div}, the proof of this theorem is along lines similar to the computation of $K_1$ of infinite vector spaces \cite[\S~5]{BK}, and we only focus on Steps I and II. 

% compute $K_1(M_S)$ for an $S$-module $M_S$ (Theorem \ref{K1semisimple}).
% \begin{thm}\label{K1semisimple}
% Let $S:=\prod_{i=1}^k S_i$ be a unital semisimple ring, where for each $i$, $S_i=M_{q_i}(R_i)$ for a division ring $R_i$ and $q_i\geq 1$. Write a right $S$-module $M_S$ as $\prod_{i =1}^k{(M_i)}_{M_{q_i}(R_i)}$, where $M_i:=Me_i$ is a right $M_{q_i}(R_i)$-module. Assume that each $M_i$ is infinite. Then $$K_1(M_S)  \cong \prod_{i=1}^kK_1\left({(M_i)}_{M_{q_i}(R_i)}\right).$$ 
% \end{thm}

\begin{rmk}\label{theoryclosedss}
Since each $M_i$ is infinite, Proposition \ref{TequalsTaleph0} and Remark \ref{chad} together yield that the theory of the module $M_S$ is closed under products.
\end{rmk}

Combining Remark \ref{composemi} and Theorem \ref{matdiv}, we obtain the following.
\begin{thm}\label{ppsemi}
Using the notations and hypotheses of Theorem \ref{K1semisimple}, $\overline{\mathcal{X}}^*(M_S) \cong \prod_{i=1}^kq_i\mathbb N \cong \mathbb N^k$.
\end{thm}

\begin{cor}[Model-theoretic $K_0$ commutes with products]
Using the notations and hypotheses of Theorem \ref{K1semisimple}, we have $K_0({M}_S)\cong K_0(M_{S_1})\otimes_\mathbb Z\cdots\otimes_\mathbb Z K_0(M_{S_k})\cong\mathbb Z[X_1,\cdots,X_k]$.
\end{cor}

The monoid $\mathbb N^k$ is naturally equipped with a partial order where $\overline m\leq\overline m'$ if $m_i\leq m'_i$ for each $1\leq i\leq k$. Let $\overline 1:=(1,1\cdots,1)\in\mathbb N^k$ and $m\cdot\overline 1$ for the constant tuple $(m,\cdots,m)\in\mathbb N^k$.

Set $\C:= \C(M_S)$ and $\C_i:= \C({(M_i)}_{S_i})$ for every $1\leq i\leq k$. Let $\pi_i:\C\to\C_i$ be the natural projection functors. For $E\in\C_i$, let $\dim_i(E)$ denote $\dim_{\C_i}(E)$ as defined in the previous sections. The above theorem forces us to assign a ``dimension vector" to definable sets in $\C.$

\begin{defn}\label{dimvec}
    For $E \in \C$, define $\dim_\C(E) :=\begin{cases}-\infty&\text{if }E=\emptyset;\\ 
    (\dim_i(\pi_i(E)))_{i=1}^k&\text{otherwise.}\end{cases}$
\end{defn}
For simplicity, we denote $\dim_\C(E)$ by $\dim(E)$ and write its $i^{th}$ component as $(\dim(E))_i$. The definition of dimension of automorphisms remains the same as in Definition \ref{supp}. For a nonempty $E\in\C$ and $\overline m \in \mathbb N^k$, let $\Omega_{\overline m}(E):=\{f\in\Aut_{\C}(E):\dim(f)\leq \overline m\}$ be the subgroup of $\Aut_{\C}(E)$ of automorphisms fixing all elements of $E$ outside a subset of dimension at most $\overline m$.

We wish to prove the following vector analogue of Proposition \ref{autinftydim} and \cite[Proposition~5.4]{BK}.
\begin{prop}\label{autinftydimss}
If $\overline{m}\in\mathbb N^k$, $E_1,E_2\in\C$ and $\dim(E_l)\geq\overline m+\overline 1$ for $l=1,2$, then $\Omega_{\overline m}(E_1)\cong\Omega_{\overline m}(E_2)$.
\end{prop}

The proof of \cite[Proposition~5.4]{BK} can be readily adapted to obtain a proof of the above except for its first line, which is the content of Lemma \ref{defbij}. (Recall that for the case of PIDs, \cite[Lemma~6.4]{BK} plays the same role as Lemma \ref{defbij}.) 

Before stating and proving the lemma, we need some technical details about special definable sets called `blocks'. Recall from \cite[Definition and Lemma~3.1.6]{Kuber1} that $B\in\C$ is a \textit{block} if $B=A\setminus\bigcup_{j=1}^t A_j$ for some $A,A_j\in\overline {\mathcal L}(M _S)$ with $A_j\subsetneq A$ for each $1\leq j\leq t$. In view of Remark \ref{theoryclosedss}, \cite[Remark~3.1.7]{Kuber1} yields that each block is nonempty.

\begin{rmk}\label{blockdim}
It follows from \cite[Lemma~2.5.7]{Kuber1} that every set in $\C$ can be written as a finite disjoint union of blocks. Suppose $B=A\setminus\bigcup_{j=1}^tA_j,B'=A\setminus\bigcup_{j=1}^{t'}A'_j\in\C$ are blocks. Then $\pi_i(B)$ is a block in $\C_i$ for each $1\leq i\leq k$. Moreover, $B\cap B'$ is non-empty, and $\dim(B)=\dim(A)=\dim(B')$.
\end{rmk}

Now we are ready to state and prove the anticipated lemma.
\begin{lem}\label{defbij}
Let $D_1,D_2\in\C$ and $\overline m=(m_1,\cdots,m_k)\in\mathbb N^k$. If $\dim(D_l)\geq\overline m+\overline 1$ for $l=1,2$, then there exists $D\in\C$ with a definable bijection $g:D\to D_2$ such that $\dim(D_1\cap D)\geq\overline m+\overline 1$.
\end{lem}

\begin{proof}
Suppose $D_1^{1}:=D_1=\bigsqcup_{p=1}^{t_1} B^{1}_{p}$ is a decomposition of $D_1$ into blocks guaranteed by Remark \ref{blockdim}. Then there exists some $p\in[t_1]$ such that $\dim(\pi_1(B^{1}_p))=(\dim(D_1))_1$. Thus, there exist $a_i\in\pi_i(D_1^1)$ for $i\neq 1$ such that $C_1^1:=\{(x,a_2,\cdots,a_n)\mid x\in \pi_1(B_p^1)\}\subseteq D_1^1$. Note that $\dim(C_1^1)=((\dim(D_1))_1,0,\cdots,0)$.

Successively incrementing $i$ from $1$ to $k-1$ in steps of size $1$, we set $D_1^{i+1}:=D_1^i\setminus C_1^i$ and repeat the process to obtain $C_1^{i+1}$ for each $i\in\{1,\cdots,k-1\}$. The construction ensures that $C_1^{i_1}\cap C_1^{i_2}=\emptyset$ for $1\leq i_1< i_2\leq k$, $C_1:=\bigsqcup_{i=1}^k C_1^i\subseteq D_1$ and $\dim(C_1)=\dim(D_1)$. We can also obtain $C_2:=\bigsqcup_{i=1}^k C_2^j\subseteq D_2$ in a similar manner.

Assume without loss of generality that $m_1^1:=(\dim(C_1^1))_1\leq (\dim(C_2^1))_1=:m_2^1$. Since $\pi_1(C_l^1)$ is a block in $\C_1$ for $l=1,2$ by Remark \ref{blockdim}, there is a definable embedding of $\pi_1(C_l^1)$ into $M_1^{m_l^1}$, say with image ${C'}_l^1$. Clearly $\overline x\mapsto(\overline x,\overline 0)$ defines an embedding $i_1:M_1^{m_1^1}\rightarrowtail M_1^{m_2^1}$ with a splitting, say $p_1$. Then $E_1:={C'_1}^1\cap p_1({C'_2}^1)$ satisfies $\dim(E_1)=m_1^1>m_1$ thanks to the final statement of Remark \ref{blockdim}. Note that $\{(x,a_2,\cdots,a_n)\mid x\in E_1\}\subseteq C_1^1$. Analogously, there is an inclusion of $E_1$ into $C_2^1$.

Repeating the above process for each $1\leq i\leq k$, we obtain $E_i$ that embeds into both $C_1^i$ and $C_2^i$, and satisfies $\dim(E_i)>m_i$. Therefore, the images of $E:=\bigsqcup_{i=1}^k\{(0,\cdots,0,x_i,0,\cdots,0)\mid x_i\in E_i\}$ in $D_1$ and $D_2$ are isomorphic and  $\dim(E)\geq\overline m+\overline 1$. By replacing the image of $E$ in $D_2$ by the image of $E$ in $D_1$, we get the required arrow $g$.
\end{proof}

%In other words, while working with definable subsets of infinite modules over semisimple rings instead of dimension, we define a dimension vector, which componentwise acts as dimension of definable subsets over $({M_i})^q_{M_{q_i}(R_i)}$, and assigning this dimension vector is justified because we have the following: 
%If $D \in \C,$ then there are $m_j\geq 1$ such that $\pi_j(D)\subseteq M_j^{m_j}$. In fact, there are natural embeddings $$\Aut_{\C}(D) \hookrightarrow \Aut_{\C}(\pi_1(D) \times \pi_2(D)) \hookrightarrow \Aut_{\C}(({M_1}^{m_1} \times {M_2}^{m_2}) \hookrightarrow \Aut_{\C}(({M_1} \times {M_2})^m),$$ where $m:= \max\{m_1,m_2\}$. 
 %TO  STRONGLY JUSTIFY
%PROPOSITION 5.4 (BK) TO BE REWRITTEN

% \begin{defn}\label{diag}
% Let $M_1$ and $M_2$ be right $R_1$ and $R_2$-module respectively. We say that $M_1\times M_2$ is a right $R_1\times R_2$ module with respect to the componentwise action when $(r_1,r_2).(m_1,m_2) = (r_1.m_1,r_2,m_2)$ and $r_1.m_2 =0 = r_2.m_1 $ for all $(r_1,r_2)\in R_1 \times R_2$, $(m_1,m_2) \in M_1\times M_2$ 
% \end{defn}
% \begin{rmk}
%     This definition emphasizes the fact that even if there is a unital ring homomorphism from $R_1 \rightarrow R_2$, the action of $R_1$ over a right $R_2$-module $M_2$ is always chosen to be trivial.
% \end{rmk}
% Using the Definition \ref{diag}, we record the following lemma that will be useful in the upcoming computations.

Given $\overline m:=(m_1,\cdots,m_k)\in\mathbb N^k$, define $M^{\overline{m}}:=\prod_{i=1}^kM_i^{m_i}$.
\begin{rmk}\label{cofinalss}
If $D \in \C,$ then there are $l_i\geq 1$ such that $\pi_i(D)\subseteq M_i^{l_i}$, and hence there are natural embeddings $$\Aut_{\C}(D) \hookrightarrow \Aut_{\C}\left(\prod_{i=1}^k\pi_i(D)\right) \hookrightarrow \Aut_{\C}\left(M^{\overline l}\right) \hookrightarrow \Aut_\C(M^{l\cdot\overline 1}),$$ where $l:= \max\{l_1,l_2\cdots l_k\}$. This observation yields that $(M^{n\cdot\overline 1})_{n\geq 1}$ is a cofinal sequence in $\C$, and hence thanks to  Remark \ref{ctblcofinal}, we have $$K_1(M_S)\cong\varinjlim_{n\in\mathbb N}\Aut_{\C}\left(M^{n\cdot\overline 1}\right).$$
\end{rmk}

For all $n \in \mathbb N$, set $\Omega_n^n(\C):=\Aut_{\C}(M^{n\cdot\overline 1})$. For each $0\leq m<n$, let $\Omega^n_m(\C):=\Omega_{m\cdot\overline 1}(M^{n\cdot\overline 1})$ (Proposition \ref{autinftydimss}). For each $n\geq 1$, let $\Upsilon^n(\C)$ denote the subgroup of $\Aut_{\C}(M^{n\cdot\overline 1})$ consisting only of $pp$-definable bijections. The notations in this paragraph are also used when we replace $M$ with $M_i$, and $\C$ with $\C_i$.
\begin{rmk}\label{ppss}
Since $S$ is a von Neumann regular ring, Remark \ref{vNRop} yields that the theory of $M_S$ eliminates quantifiers. As a result, $\Upsilon^n(\C)$ is the group of definable $S$-linear automorphisms of $M^{n\cdot\overline 1}$. In other words, $\Upsilon^n(\C)\cong GL_n(S^{op})\ltimes M^{n\cdot\overline 1}$, , where $GL_n(S^{op})$ acts on $M^{n\cdot\overline 1}$ on the right by matrix multiplication.
\end{rmk}
\begin{lem}\label{smprd}
Using the notations and hypotheses of Theorem \ref{K1semisimple}, we have$$GL_n\left(S^{op}\right)\ltimes M^{n\cdot\overline 1}\cong \prod_{i=1}^k \left(GL_n(S_i^{op})\ltimes M_i^n \right).$$
\end{lem}
\begin{proof}
The proof readily follows from the observation that $GL_n(S^{op})\cong \prod_{i=1}^k GL_n(S_i^{op})$ acts on $M^{n\cdot\overline 1}=(\prod_{i=1}^k M_i)^n$ componentwise.
\end{proof}
Combining Remark \ref{ppss} with the lemma above, we get the following conclusion.
\begin{cor}\label{smprdpp}
For each $n \in \mathbb N$, $\Upsilon^n(\C)\cong \prod_{i=1}^k \Upsilon^n(\C_i) $. 
\end{cor}

We have a chain of normal subgroups of $\Omega^{n}_{n}(\C)$:
\begin{equation}
\Omega^{n}_0(\C)\lhd\Omega^{n}_1(\C)\lhd\cdots\lhd\Omega^{n}_{n-1}(\C)\lhd\Omega^{n}_{n}(\C).
\end{equation}
The group $\Upsilon^n(\C)$ acts on $\Omega^n_{n-1}(\C)$ by conjugation and $\Omega^n_n(\C)=\Upsilon^n(\C)\ltimes\Omega^n_{n-1}(\C)$. Fix some $0<m<n$. Let $\C_{m\cdot\overline 1}(M^{n\cdot\overline 1})$ denote the full subcategory of $\C$ consisting of definable subsets of $M^{n\cdot\overline 1}$ of dimension at most $m\cdot\overline 1$. The restriction of $\sqcup$ equips $\C_{m\cdot\overline 1}(M^{n\cdot\overline 1})$ with a symmetric monoidal structure and $\Omega^n_m(\C)\cong\Aut(\C_{m\cdot\overline 1}(M^{n\cdot\overline 1}))$. We want to find a subgroup $\Upsilon^n_m(\C)$ of $\Omega^n_m(\C)$ such that $\Omega^n_m(\C)=\Upsilon^n_m(\C)\ltimes\Omega^n_{m-1}(\C)$.

Let $\Sigma_j$ denote the permutation group of a finite set of size $j$, and $\overline\Sigma$ denote the finitary permutation group of a countably infinite set.

For $\overline j:=(j_1,\cdots,j_k)\in\mathbb N^k$, let $S_{m,\overline j}\in\C_{m\cdot\overline 1}(M^{n\cdot\overline 1})$ denote a copy of $\prod_{i=1}^k(\bigsqcup_{l=1}^{j_i}M_i^m)$ in such a way that $S_{m,\overline j}\subseteq S_{m,\overline j'}$ whenever $\overline j\leq\overline j'$ in $\mathbb N^k$. Note that
\begin{equation*}
\Aut_\C(S_{m,\overline 1})\cong\Omega^m_m(\C)\cong\Upsilon^m(\C)\ltimes\Omega^m_{m-1}(\C)\cong\Upsilon^m(\C)\ltimes\Omega^n_{m-1}(\C),
\end{equation*}
where the action of $\Upsilon^m(\C)$ on $\Omega^n_{m-1}(\C)$ is induced by the isomorphism $\Omega^m_{m-1}(\C)\cong\Omega^n_{m-1}(\C)$ given by Proposition \ref{autinftydimss}. For similar reasons, we also have
\begin{equation*}
\Aut_\C(S_{m,\overline j})\cong(\Upsilon^m(\C)\wr\prod_{i=1}^k\Sigma_{j_i})\ltimes\Omega^m_{m-1}(\C)\cong(\Upsilon^m(\C)\wr\prod_{i=1}^k\Sigma_{j_i})\ltimes\Omega^n_{m-1}(\C),
\end{equation*}
where the group $(\Upsilon^m(\C)\wr\prod_{i=1}^k\Sigma_{j_i})$ acts on $\Omega^m_{m-1}(\C)$ by conjugation and permutes lower dimensional subsets of $S_{m,\overline j}\subset M^{n\cdot\overline 1}$. Since $(S_{m,j\cdot\overline 1})_{j\in\mathbb N}$ is a cofinal sequence in $\C_{m\cdot\overline 1}(M^{n\cdot\overline 1})$, Remark \ref{ctblcofinal} yields
\begin{eqnarray*}
\Omega^n_m(\C)&\cong&\varinjlim_{j\in \mathbb N}\Aut_\C(S_{m,j\cdot\overline 1})\\
&\cong&\varinjlim_{j\in \mathbb N}\left((\Upsilon^m(\C)\wr\prod_{i=1}^k\Sigma_j)\ltimes\Omega^n_{m-1}(\C)\right)\\
&\cong&\left(\varinjlim_{j\in \mathbb N}(\Upsilon^m(\C)\wr\prod_{i=1}^k\Sigma_j)\right)\ltimes\Omega^n_{m-1}(\C)\\
&\cong&\left(\Upsilon^m(\C)\wr\prod_{i=1}^k\overline\Sigma\right)\ltimes\Omega^n_{m-1}(\C).
\end{eqnarray*}
Let $\Upsilon^{n}_{m}(\C):=\Upsilon^{m}(\C)\wr\prod_{i=1}^k\overline\Sigma$. Note that $\Upsilon^{n}_{m}(\C)$ acts on $\Omega^{n}_{m-1}(\C)$ by conjugation.
Thus each $\Omega^n_n(\C)$ is an iterated semi-direct product of certain wreath products as in the following expression.
$$\Omega^n_n(\C)\cong\Upsilon^n(\C)\ltimes(\Upsilon^n_{n-1}(\C)\ltimes(\Upsilon^n_{n-2}(\C)\ltimes(\cdots(\Upsilon^n_1(\C)\ltimes\Omega^n_0(\C))\cdots))).$$

Thanks to Corollary \ref{smprdpp} and the fact that all conjugation actions in the above expression are componentwise, we conclude $\Omega^n_n(\C)\cong\prod_{i=1}^k\Omega^n_n(\C_i)$. This completes the proof of Theorem \ref{K1semisimple}.

% \begin{rmk}\label{algK1commuteswithproducts}
% Theorem \ref{K1semisimple} is a partial model-theoretic analogue of algebraic $K_1$ \cite[Example~III.1.1.3]{Weibel} which states that algebraic $K_1$ commutes with finite products. 
% \end{rmk}

Since $\mathrm{Free}(S)$ is cofinal in $\prod_{i=1}^k\mathrm{Free}(S_i)$ (Remark \ref{cofinalss}) and algebraic $K_1$ commutes with finite products, we obtain the following consequence of Theorems \ref{algmoddiv}, \ref{K1Final} and \ref{K1semisimple}.
\begin{cor}\label{algmodss}
If $S=\prod_{i=1}^kS_i$ is a semisimple ring with each $S_i$ infinite, then there is a natural embedding of $K_1^\oplus(S)$ into $K_1(S_S)$ given by a $k$-tuple of maps as in Remark \ref{algmoddivrmk}.
\end{cor}
 
The proof of Theorem \ref{K1semisimple} follows verbatim to yield the following stronger result.
\begin{thm}\label{partition}
For $1\leq i\leq k$, let $S_i$ be a von Neumann regular ring and $M_i$ be a right $S_i$-module satisfying the hypotheses of Remark \ref{step12}. Suppose $\prod_{i=1}^kS_i$ acts on $\prod_{i=1}^kM_i$ componentwise. Then $$K_1\left(\left(\prod_{i=1}^kM_i\right)_{\left(\prod_{i=1}^kS_i\right)}\right)\cong \prod_{i=1}^kK_1({(M_i)}_{S_i}).$$
\end{thm}

If the hypotheses of Remark \ref{step12} fail for a von Neumann regular ring $R$, then our recipe fails to compute the group $K_1(M_R)$.

\subsection*{Declaration of competing interest}

The authors declare that they have no known competing financial interests or personal relationships that could have appeared to influence the work reported in this paper.

\subsection*{Acknowledgments:} The authors thank Mike Prest for pointing out the result regarding elimination of imaginaries for modules over von Neumann regular rings.

\subsection*{Data availability}
No data was used for the research described in the article.

\bibliographystyle{alpha}
\bibliography{main}

@article {Kuber1,
    AUTHOR = {Kuber, Amit},
     TITLE = {Grothendieck rings of theories of modules},
   JOURNAL = {Ann. Pure Appl. Logic},
  FJOURNAL = {Annals of Pure and Applied Logic},
    VOLUME = {166},
      YEAR = {2015},
    NUMBER = {3},
     PAGES = {369--407},
      ISSN = {0168-0072,1873-2461},
   MRCLASS = {03C60 (16B70 16E20 55U05)},
  MRNUMBER = {3292810},
MRREVIEWER = {Gena\ Puninski},
}

@article {Baur,
    AUTHOR = {Baur, Walter},
     TITLE = {Elimination of quantifiers for modules},
   JOURNAL = {Isr. J. Math.},
  FJOURNAL = {Israel Journal of Mathematics},
    VOLUME = {25},
      YEAR = {1976},
    NUMBER = {1-2},
     PAGES = {64--70},
      ISSN = {0021-2172},
   MRCLASS = {02H15},
  MRNUMBER = {457194},
MRREVIEWER = {O.\ V.\ Belegradek},
}

@article {Kra,
    AUTHOR = {Kraj\'{\i}\v{c}ek, Jan and Scanlon, Thomas},
     TITLE = {Combinatorics with definable sets: {E}uler characteristics and
              {G}rothendieck rings},
   JOURNAL = {Bull. Symb. Log.},
  FJOURNAL = {The Bulletin of Symbolic Logic},
    VOLUME = {6},
      YEAR = {2000},
    NUMBER = {3},
     PAGES = {311--330},
      ISSN = {1079-8986,1943-5894},
   MRCLASS = {03C07 (03C62 03F20 68Q15)},
  MRNUMBER = {1803636},
MRREVIEWER = {M.\ Yasuhara},
}

@book{Weibel,
    author = {Weibel, Charles },
    title = {The K-book: An introduction to algebraic $K$-theory},  
    publisher = {Graduate Studies in Math., vol. 145, AMS},
    year ={2013} 
}

@book{PreBk,
    author = {Prest, Mike },
    title = { Model Theory and Modules, London Math. Soc., Lecture Notes Ser., vol. 130}, publisher = { Cambridge University Press},
    year ={1988} 
}

@book{Bass,
    author = {Bass, Hyman },
    title = {Algebraic $K$-theory},  
    publisher = {W.A. Benjamin, Inc},
    year ={1968} 
}

@book {Rot,
    AUTHOR = {Rotman, Joseph J.},
     TITLE = {An introduction to homological algebra},
    SERIES = {Universitext},
   EDITION = {Second},
 PUBLISHER = {Springer, New York},
      YEAR = {2009},
     PAGES = {xiv+709},
      ISBN = {978-0-387-24527-0},
   MRCLASS = {18Gxx (13Dxx 16Exx 18-01 20J06)},
  MRNUMBER = {2455920},
MRREVIEWER = {Fernando\ Muro},
       DOI = {10.1007/b98977},
       URL = {https://doi.org/10.1007/b98977},
}

@book{PSL,
    author = {Prest, Mike },
    title = {Purity, Spectra and Localization, Encyclopedia of Mathematics
    and its Applications, vol. 121}, publisher = {Cambridge University Press},
    year ={2009} 
}

@article {BK,
    AUTHOR = {Banerjee, Sourayan and Kuber, Amit},
     TITLE = {Model-theoretic {$K_1$} of free modules over {PID}s},
   JOURNAL = {Ann. Pure Appl. Logic},
  FJOURNAL = {Annals of Pure and Applied Logic},
    VOLUME = {176},
      YEAR = {2025},
    NUMBER = {9},
     PAGES = {103613},
}

@book {Hodges,
    AUTHOR = {Hodges, Wilfrid},
     TITLE = {Model theory},
    SERIES = {Encyclopedia of Mathematics and its Applications},
    VOLUME = {42},
 PUBLISHER = {Cambridge University Press, Cambridge},
      YEAR = {1993},
     PAGES = {xiv+772},
}

\end{document}